\documentclass[12pt]{amsart}
\usepackage{graphicx}
\usepackage{amscd}
\usepackage[top=1in,bottom=1in,left=1.15in,right=1.15in]{geometry} 
\usepackage{amsmath,amssymb,amsfonts}
\usepackage{amsthm}
\usepackage{color}
\usepackage{hyperref}
\usepackage{tikz}
\newtheorem{theo}{Theorem}[section]
\newtheorem{prop}[theo]{Proposition}
\newtheorem{lemma}[theo]{Lemma}
\newtheorem{defn}[theo]{Definition}
\newtheorem{rem}[theo]{Remark}

\newtheorem{obs}[theo]{Observation}

\begin{document}
\title{On knots having zero negative unknotting number}
\author{Yuanyuan Bao}
\address{
Graduate school of Mathematics,
Tokyo Institute of Technology,
Oh-okayama, Meguro, Tokyo 152-8551, Japan
}
\email{bao@ms.u-tokyo.ac.jp}
\date{}
\begin{abstract}
A knot in the 3-sphere is said to have zero negative unknotting number if it can be transformed into the unknot by performing only positive crossing changes. In this paper, we provide an obstruction for a knot to having zero negative unknotting number, and discuss its application to two classes of knots.

\end{abstract}
\thanks{The author is supported by Grant-in-Aid for JSPS Fellows.}
\maketitle

\section{Introduction}
Let $K$ be an oriented knot in the 3-sphere $S^{3}$. It is well known that $K$ can be transformed into the unknot by crossing changes, which are local moves illustrated in Figure~\ref{fig:f1}. By counting the signs of the crossing changes, we can define the set $\mathcal{U}(r,s)$ to be the set of all knots that can be unknotted by performing $r$ positive crossing changes and $s$ negative crossing changes (Cochran and Lickorish \cite{MR849471}). 
\begin{figure}[h]
	\centering
		\includegraphics[width=0.6\textwidth]{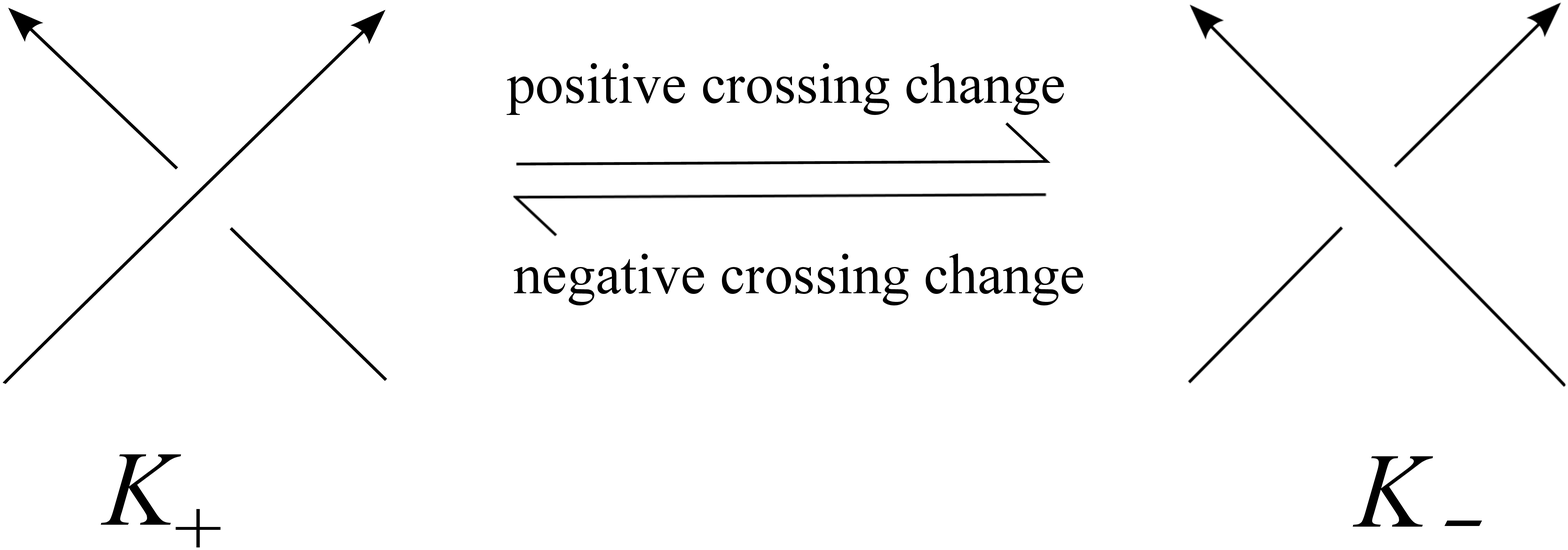}
	\caption{}
\label{fig:f1}
\end{figure}

Define $u_{-}(K)=\min\lbrace s \vert K \in \mathcal{U}(r,s)\rbrace$, and call it the \textit{negative unknotting number} of $K$. Similarly we can define the positive unknotting number $u_{+}(K)$. Sometimes we use $u_{\pm}(K)$ to denote both the case of $u_{-}(K)$ and $u_{+}(K)$. The signed unknotting number has been studied for a long time. In this paper, we provide an obstruction for a knot to having zero negative unknotting number. Before stating the result, let us review some necessary terminologies.

Given $p/q\in\mathbb{Q}\setminus\lbrace 0\rbrace$, consider a continued fraction of $p/q=[a_{1}, a_{2}, \cdots, a_{n}]$. Namely,
$$p/q=a_{1}-\cfrac{1}{a_{2}-\cfrac{1}{a_{3}-\cdots-\cfrac{1}{a_{n}}}}.$$
Let $\Omega$ be the set of $p/q\in\mathbb{Q}\setminus\lbrace 0\rbrace$ which has a continued fraction $[a_{1}, a_{2}, \cdots, a_{n}]$ such that all but at most two of the $a_{i}$ satisfy that
 $a_{i}\leq -2$ for $2 \leq i \leq n-1$ and $a_{i}\leq -1$ for $i=1, n$. In particular, any rational number $-p/q$ with $p\geq q>0$ is an element in $\Omega$. We can see this from the following arguments. Require that $p$ and $q$ are relatively prime integers. When $p=1$ the statement is true. When $p\geqslant 2$ we show that $p/q$ has a continued fraction $[b_{1}, b_{2}, \cdots, b_{n}]$ with $b_{i}\geqslant 2$ . We do induction on $q$. If $q=1$, then clearly such a continued fraction exists. Suppose now that the continued fraction exists for any rational number $r/s$ with $r>s>0$ and $q>s$. For $p/q$, there exists an integer $m\geqslant 1$ so that $mq < p < (m+1)q$. Here we have $\dfrac{p}{q}=(m+1)-\dfrac{1}{q/[(m+1)q-p]}$ with $m+1\geqslant 2$ and $q>(m+1)q-p>0$. By induction, we see that $q/[(m+1)q-p]$ has a continued fraction $[b_{1}, b_{2}, \cdots, b_{n}]$ with $b_{i}\geqslant 2$. Then $[m+1, b_{1}, b_{2}, \cdots, b_{n}]$ is a continued fraction for $p/q$. This implies that $[-(m+1), -b_{1}, -b_{2}, \cdots, -b_{n}]$ is a continued fraction for $-p/q$, which makes $-p/q \in \Omega$.

Let $Y$ be a compact oriented rational homology 3-sphere and $\mathrm{Spin}^{c}(Y)$ be the set of its spin$^{c}$-structure. Given $s\in \mathrm{Spin}^{c}(Y)$, Ozsv{\'a}th and Szab{\'o} \cite{MR1957829} defined a $\mathbb{Q}$-valued invariant $d(Y, s)$ for the pair $(Y, s)$, called correction term or $d$-invariant. We will discuss the definition of $d$-invariant and its properties in Section~2. 
Our main result is as follows.

\begin{theo}
\label{main}
Let $S^{3}_{p/q}(K)$ denote the 3-manifold obtained via Dehn surgery of $S^{3}$ along $K$ with slope $p/q$, and choose the convention that $S^{3}_{p/q}(O)\cong L(p,q)$ where $O$ stands for the unknot and $L(p,q)$ is the $(p,q)$-lens space. Then we have the following properties about the $d$-invariants for $S^{3}_{p/q}(K)$:
If $u_{-}(K)=0$, then there exists an affine isomorphism $\alpha : \mathrm{Spin}^{c}(L(p,q))\rightarrow \mathrm{Spin}^{c}(S^{3}_{p/q}(K))$ such that $d(S^{3}_{p/q}(K), \alpha(s))=d(L(p,q), s)$ for any $p/q \in \Omega$ and $s\in \mathrm{Spin}^{c}(L(p,q))$. 
%If $K$ is smoothly concordant to a knot $K'$, then there is an affine isomorphism $\beta: \mathrm{Spin}^{c}(S^{3}_{p/q}(K))\rightarrow \mathrm{Spin}^{c}(S^{3}_{p/q}(K'))$ such that $d(S^{3}_{p/q}(K), s)=d(S^{3}_{p/q}(K'), \beta (s))$ for any $p/q \in \mathbb{Q}\setminus\lbrace 0\rbrace$ and $s\in \mathrm{Spin}^{c}(S^{3}_{p/q}(K))$.

\end{theo}

Let $A$ (respectively $T$, $S$ and $R$) be the set of algebraically slice knots (respectively topologically slice knots, smoothly slice knots and ribbon knots) in $S^{3}$. Then they satisfy the relation $R\subseteq S \subset T \subset A$.
Referring to \cite{MR0353295}, we consider the following sets of knots. Define $u^{A}(r,s)$ (respectively $u^{T}(r,s)$, $u^{S}(r,s)$ and $u^{R}(r,s)$) to be the set of knots that can be transformed into knots in $A$ (respectively $T$, $S$ and $R$) by performing $r$ positive crossing changes and $s$ negative crossing changes. Given a knot $K$, we can define $u^{A}_{\pm}(K)$, $u^{T}_{\pm}(K)$, $u^{S}_{\pm}(K)$ and $u^{R}_{\pm}(K)$, as we did for $u_{\pm}(K)$. Then these invariants satisfy the relation $$u^{A}_{\pm}(K)\leq u^{T}_{\pm}(K) \leq u^{S}_{\pm}(K) \leq u^{R}_{\pm}(K) \leq u_{\pm}(K)$$ for any knot $K$ in $S^{3}$.

There have been many tools and invariants which can be applied to detect whether a knot has zero negative unknotting number or not. The first class of invariants we can use are signature $\sigma(\cdot)$, Rasmussen invariant $s(\cdot)$ and Ozsv{\'a}th and Szab{\'o}'s invariant $\tau(\cdot)$, which satisfy the relation $\nu (K_{-})\leq \nu (K_{+})\leq \nu (K_{-}) +2$ with $\nu$ standing for either $\sigma$, $s$ or $2\tau$. A knot $K$ with $u_{-}(K)=0$ has the property that $\nu (K)\geq 0$ (see also \cite[Corollary~3.9]{MR849471}). The signature $\sigma$, as an algebraical concordance invariant, does not distinguish $u_{\pm}$ from $u^{A}_{\pm}$, while $s$ and $\tau$, as smooth concordance invariants, do not tell the difference between $u_{\pm}$ and $u^{S}_{\pm}$. Another useful tool is Donaldson's diagonalization theorem, which was applied by Cochran and Lickorish \cite{MR849471} to show that the untwisted Whiteheand double of a class of knots have non-zero negative unknotting numbers. Some ingredients from Heegaard Floer homology were applied as well recently. Let $\Sigma (K)$ denote the double-branched cover of $S^{3}$ along $K$. Owens \cite{MR2388097, owens} used the $d$-invariants of $\Sigma (K)$ to study the signed unknotting number of $K$. These techniques do not distinguish $u_{\pm}$ from $u^{S}_{\pm}$ as well.

Two basic ways to obtain 3-manifolds from a knot are doing Dehn surgeries and taking branched coverings of $S^{3}$ along the knot. Comparing with Owens' work, Theorem \ref{main} is an application of $d$-invariant in the other direction.  However, our obstruction does not distinguish $u_{\pm}$ from $u^{S}_{\pm}$ either, due to the following observation. 

\begin{obs}
If $K$ is smoothly concordant to a knot $K'$, then there is an affine isomorphism $\beta: \mathrm{Spin}^{c}(S^{3}_{p/q}(K))\rightarrow \mathrm{Spin}^{c}(S^{3}_{p/q}(K'))$ so that $d(S^{3}_{p/q}(K), s)=d(S^{3}_{p/q}(K'), \beta (s))$ for any $p/q \in \mathbb{Q}\setminus\lbrace 0\rbrace$ and $s\in \mathrm{Spin}^{c}(S^{3}_{p/q}(K))$.
\end{obs}

Section~2 is devoted to the proof of Theorem~\ref{main} and applications of the theorem.

\section{Proof of Theorem~\ref{main}}
The definition and some properties of $d$-invariant were introduced by Ozsv{\'a}th and Szab{\'o} in \cite{MR1957829}. Here we recall the definition first and then review those properties to be used in Section~2.2, where the proof is given. 

\subsection{The definition and some properties of $d$-invariant.}

Let $Y$ be an oriented rational homology 3-sphere and $s\in\mathrm{Spin}^{c}(Y)$ be a spin$^{c}$-structure over $Y$. Ozsv{\'a}th and Szab{\'o} \cite{MR2113019} defined the Heegaard Floer homology associated with the pair $(Y, s)$. One version of this homology is $HF^{+}(Y, s)$, a graded torsion ${\mathbb F}[U]$-module with ${\mathbb F}:={\mathbb Z}/2{\mathbb Z}$ and $U$ a formal polynomial variable. For more information about Heegaard Floer homology, please see the appendix.

The definition of $d$-invariant was first given in \cite{MR1957829}. In this paper, we refer to the following definition.

\begin{defn} 
Let $Y$ be an oriented rational homology 3-sphere and $s\in\mathrm{Spin}^{c}(Y)$ be a spin$^{c}$-structure over $Y$. Define
$$d(Y,s)=\min_{\xi \neq 0 \in HF^{+}(Y,s)}\lbrace \operatorname{gr}(\xi)\vert \xi\in \mathrm{Im}(U^{k}), \text{ for all } k\geq 0\rbrace,$$ where $\operatorname{gr}(\xi)$ denotes the grading of $\xi$ and takes value in ${\mathbb Q}$.
\end{defn}

The $d$-invariants for $Y$ and $-Y$, where $``-"$ means the reverse of the orientation, are related by the formula 
\begin{equation}
\label{5}
d(-Y,s)=-d(Y,s)
\end{equation}
under the natural identification ${\rm Spin}^{c}(Y)\cong {\rm Spin}^{c}(-Y)$.

It is known that $d$-invariant is additive under the connected sum of 3-manifolds. Given two rational homology 3-spheres $Y_{1}$ and $Y_{2}$, and spin$^{c}$-structures $s_{i}\in\mathrm{Spin}^{c}(Y_{i})$ for $i=1,2$, we have 
\begin{equation}
\label{sum}
d(Y_{1}\sharp Y_{2}, s_{1}\sharp s_{2})=d(Y_{1}, s_{1})+d(Y_{2}, s_{2}),
\end{equation}
where $s_{1}\sharp s_{2}$ means the sum of spin$^{c}$-structures.

Suppose that $Y$ is an oriented rational homology 3-sphere, that $X$ is an oriented negative-definite simply connected smooth 4-manifold with $\partial X=Y$ and that $t\in {\rm Spin}^{c}(X)$. Then it was shown in \cite{MR1957829} that 
\begin{eqnarray}
\label{1}
d(Y, t\bigm|_Y) &\geq & \frac{c_{1}^{2}(t)+b_{2}(X)}{4},
\end{eqnarray}
where $c_{1}(t)\in H^{2}(X; \mathbb{Z})$ is the first Chern class of $t$ and $c_{1}^{2}(t)\in \mathbb{Q}$ is the image of $(c_{1}(t), c_{1}(t))$ under the cup product, $b_{2}(X)$ is the second Betti number of $X$, and $t\bigm|_Y$ is the restriction of $t$ onto $Y$. Suppose $X$ is a rational smooth 4-cobordism from the rational homology 3-sphere $Y$ to the rational homology 3-sphere $Y'$. Then 
\begin{equation}
\label{4}
d(Y, t\bigm|_Y)=d(Y', t\bigm|_{Y'})
\end{equation}
for any $t\in {\rm Spin}^{c}(X)$. The above relations (\ref{5}), (\ref{sum}), (\ref{1}) and (\ref{4}) all come from Ozsv{\'a}th and Szab{\'o}'s original paper
\cite{MR1957829}.

It is in general not easy to calculate the $d$-invariant. But for some plumbed 3-manifolds, a formula for $d$-invariant was established in \cite{MR1988284}, which we recall here. Let $G$ be a tree equipped with an integer-valued function $m$ on its vertices. Such a tree $G$ gives rise to a 4-manifold $X(G)$ with boundary $Y(G)$.  The 4-manifold $X(G)$ is obtained by plumbing a collection of disk bundles over the 2-sphere indexed by the vertices of $G$, so that the Euler number of the disk bundle corresponding to a vertex $v$ is the value $m(v)$. Two disk bundles are plumbed if their corresponding vertices of $G$ are connected by an edge in $G$. The second integral homology group of $X=X(G)$ is freely generated by the vertices of $G$. Namely $H_{2}(X; \mathbb{Z})\cong \bigoplus_{v\in V(G)}\mathbb{Z}<e_v>$ where $V(G)$ is the set of vertices of $G$ and $e_{v}$ is the generator of $H_{2}(X; \mathbb{Z})$ corresponding to $v$. Then the intersection form on $H_{2}(X; \mathbb{Z})$ is given by $G$ as follows. 
\begin{equation*}
e_{v}\cdot e_{w}=
\begin{cases}
m(v) & \text{ if $v=w$}\\
1 & \text{if $i\neq j$ and $v$ and $w$ are connected by an edge}\\
0 & \text{otherwise}.
\end{cases}
\end{equation*}

A tree $G$ described above is called \textit{negative-definite} if the intersection form associated with $G$ is negative definite. 
A vertex $v\in V(G)$ is said to be a \textit{bad vertex} if $\vert m(v)\vert$ is less than the number of its adjoining edges.
Here is the formula we will use.
\begin{theo}[Ozsv{\'a}th-Szab{\'o} \cite{MR1988284} ]
\label{os}
Let $G$ be a negative-definite tree with at most two bad vertices, and fix a spin$^{c}$-structure $s\in  {\rm Spin}^{c}(Y(G))$. Then
\begin{equation}
d(Y(G), s)=\max \left\{\dfrac{c_{1}(t)^{2}+\vert V(G)\vert}{4} \left| t\in {\rm Spin}^{c}(X), t\bigm|_{Y(G)}=s \right. \right\},
\end{equation}
where $\vert V(G) \vert$ is the number of vertices of $G$ and is also the second Betti number of $X(G)$.
\end{theo}

Let $K_{+}$ and $K_{-}$ be two knots which are identical except in a neighbourhood of a crossing, as shown in Figure~\ref{fig:f1}. In Section~4 of \cite{MR2902750}, the author showed an inequality between the $d$-invariants of $S_{p/q}^{3}(K_{+})$ and those of $S_{p/q}^{3}(K_{+})$ for any $p/q \in \mathbb{Q}\setminus\lbrace 0\rbrace$. There is a cobordism $W$ from $S_{p/q}^{3}(K_{+})$ to $S_{p/q}^{3}(K_{-})$ given by attaching a 2-handle to $S_{p/q}^{3}(K_{+})\times [0,1]$ along a circle with framing -1. The cobordism $W$ provides a natural affine isomorphism $\alpha : \mathrm{Spin}^{c}(S^{3}_{p/q}(K_{-}))\rightarrow \mathrm{Spin}^{c}(S^{3}_{p/q}(K_{+}))$, and it was proved in \cite[Eq.~(4)]{MR2902750} that 
\begin{equation}
\label{3}
\begin{array}{l}
d(S_{p/q}^{3}(K_{-}), s)-d(S_{p/q}^{3}(K_{+}), \alpha (s))\geq 0\\
\end{array}
\end{equation}
for any $s\in \mathrm{Spin}^{c}(S^{3}_{p/q}(K_{-}))$. The map $\alpha$ is natural in the sense that $\alpha(m_{-}\cdot s)=m_{+}\cdot\alpha(s)$ for any $s\in \mathrm{Spin}^{c}(S^{3}_{p/q}(K_{-}))$, where $``\cdot "$ means the action of the second cohomology group on the set of spin$^{c}$-structures. Here $m_{-}$ and $m_{+}$ are the Poincar{\'e} duals of the homology classes of the coherently oriented meridians of $K_{-}$ and $K_{+}$, and thus they generate $H^{2}(S^{3}_{p/q}(K_{-}); \mathbb{Z})$ and $H^{2}(S^{3}_{p/q}(K_{+}); \mathbb{Z})$.

\begin{figure}
	\centering
		\includegraphics[width=0.8\textwidth]{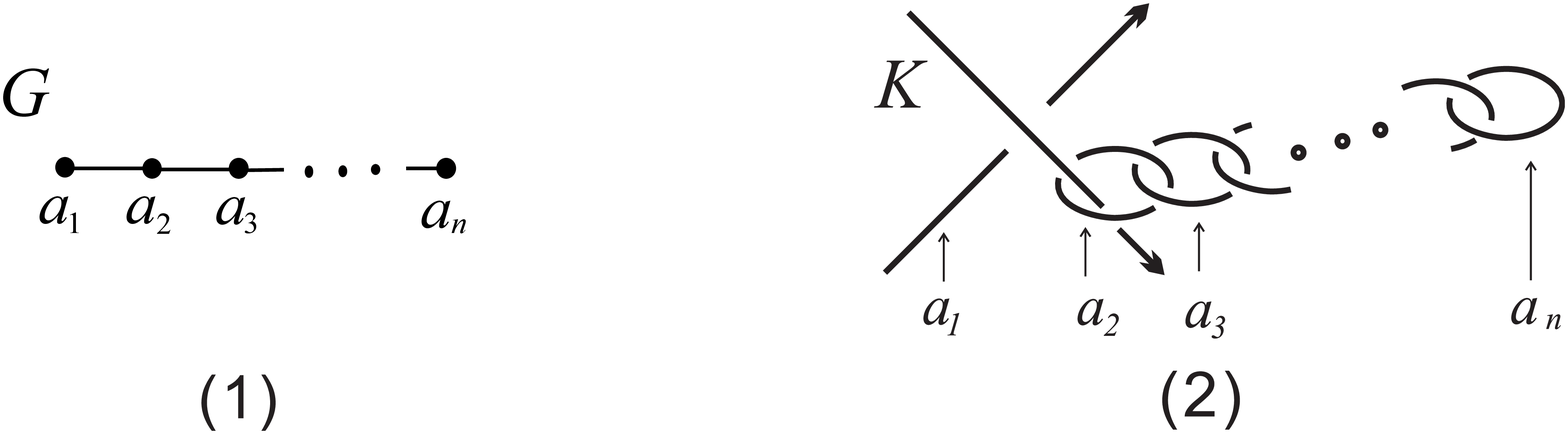}
	\caption{}
	\label{fig:f2}
\end{figure}

\subsection{Proof of Theorem~\ref{main}}

\begin{proof}[Proof of Theorem~\ref{main}]
Suppose $u_{-}(K)=0$, which means $K$ can be transformed into the unknot by performing only positive crossing changes. By (\ref{3}), we have an affine isomorphism $\alpha : \mathrm{Spin}^{c}(L(p,q))\rightarrow \mathrm{Spin}^{c}(S^{3}_{p/q}(K))$ (here we abuse the notation $\alpha$) such that  
\begin{equation*}
\label{*}
d(S_{p/q}^{3}(K), \alpha (s))\leq d(S_{p/q}^{3}(O), s)=d(L(p,q), s), \tag{$\ast$}
\end{equation*}
for any $p/q\in \Omega$ and any $s\in \mathrm{Spin}^{c}(L(p,q))$.

Let $[a_{1}, a_{2}, \cdots, a_{n}]$ be the continued fraction of $p/q$ which makes $p/q$  belong to $\Omega$. Let $G$ be the tree given by the above continued fraction, as shown in Figure~\ref{fig:f2} (1). From the definition of $\Omega$, it is easy to check that $G$ is a negative-definite tree with at most two bad vertices. The manifold $Y(G)$ in this case is the lens space $L(p,q)$. By Theorem~\ref{os}, we have
\begin{equation*}
\label{*'}
d(L(p,q), s)=\max \left\{\dfrac{c_{1}(t)^{2}+n}{4} \left| t\in {\rm Spin}^{c}(X(G)), t\bigm|_{L(p,q)}=s \right. \right\}, \tag{$\ast'$}
\end{equation*}
for any $s\in \mathrm{Spin}^{c}(L(p,q))$.

On the other hand, the 3-manifold $S_{p/q}^{3}(K)$ bounds a smooth 4-manifold $X'$ which is constructed by the Kirby diagram in Figure~\ref{fig:f2} (2). The intersection form on $X'$ is the same as that on $X(G)$, and therefore is negative definite. By (\ref{1}), we have
\begin{equation*}
\label{*''}
d(S_{p/q}^{3}(K), s)\geq\max \left\{\dfrac{c_{1}(t)^{2}+n}{4} \left| t\in {\rm Spin}^{c}(X'), t\bigm|_{S_{p/q}^{3}(K)}=s \right. \right\},\tag{$\ast''$}
\end{equation*}
for any $s\in \mathrm{Spin}^{c}(S_{p/q}^{3}(K))$.

It is easy to see that there is an affine isomorphism $\varphi :\mathrm{Spin}^{c}(X(G))\rightarrow \mathrm{Spin}^{c}(X')$ such that $c_{1}^{2}(t)=c_{1}^{2}(\varphi(t))$ and $\alpha (t\bigm|_{L(p,q)})=\varphi (t)\bigm|_{S_{p/q}^{3}(K)}$ for any $t\in \mathrm{Spin}^{c}(X(G))$. 
Then By (\ref{*'}) and (\ref{*''}), we have $d(S_{p/q}^{3}(K), \alpha (i))\geq d(L(p,q),s)$ for any $s\in \mathrm{Spin}^{c}(L(p,q))$. This together with (\ref{*}) implies that there is an affine isomorphism $\alpha : \mathrm{Spin}^{c}(L(p,q))\rightarrow \mathrm{Spin}^{c}(S^{3}_{p/q}(K))$ such that $d(S_{p/q}^{3}(K), \alpha (s))=d(L(p,q), s)$.

\end{proof}

\subsection{Application}
\subsubsection{Knots in $S^{3}$ which admit $L$-space surgeries}
In this part, we discuss an application of Theorem~\ref{main} to knots in $S^{3}$ which admit $L$-space surgeries. In particular we prove Proposition~\ref{prop}. First, we recall a theorem in \cite{MR2764036}.

Recall that an $L$-space is a rational homology 3-sphere $Y$ whose Floer homology $HF^{+}(Y,s)$ in each spin$^{c}$-structure $s\in \mathrm{Spin}^{c}(Y)$, as a relatively graded $\mathbb{F}[U]$-module, is isomorphic to $HF^{+}(S^{3})$. Let $K\subset S^{3}$ be a knot in the 3-sphere. Write its normalized Alexander polynomial as $$\Delta_{K}(T)=a_{0}+\sum_{i>0}a_{i}(T^{i}+T^{-i}),$$ and let $t_{i}(K)=\sum_{j=0}^{\infty}ja_{|i|+j}$ for $i\in \mathbb{Z}$. We have the following theorem.

%For the manifold $S_{p/q}^{3}(K)$ obtained from the knot $K$ and $p/q\in \mathbb{Q}\setminus \{0\}$, there is an affine identification $H_{1}(S_{p/q}^{3}(K); \mathbb{Z})=\mathbb{Z}/p\mathbb{Z}\cong \mathrm{Spin}^{c}(S^{3}_{p/q}(K))$ described in \cite{MR2764036}. Given an element $i\in \mathbb{Z}$, we use $s_{i}$ to denote the corresponding element in $\mathrm{Spin}^{c}(S^{3}_{p/q}(K))$. The theorem is as follows.

\begin{theo}[Ozsv{\'a}th-Szab{\'o} \cite{MR1957829}]
\label{lemma}
Let $K\subset S^{3}$ be a knot which admits an $L$-space surgery, for some integer $p\geq 1$. Then there is a bijection $f: H_{1}(S_{p}^{3}(K); \mathbb{Z})=\mathbb{Z}/p\mathbb{Z}\cong \mathrm{Spin}^{c}(S^{3}_{p}(K))$ and an affine isomorphism $g: \mathrm{Spin}^{c}(S^{3}_{p}(K))\rightarrow \mathrm{Spin}^{c}(L(p, 1))$ satisfying the following property. For all integers $i$ with $|i|\leq p/2$ we have
$$d(S^{3}_{p}(K), f([i]))-d(L(p, 1), g\circ f([i]))=-2t_{i}(K)\leq 0,$$ while for all $|i|>p/2$, we have $t_{i}(K)=0$. 
\end{theo}

We prove the following property for knots which admit $L$-space surgeries.

\begin{prop}
\label{prop}
Suppose the knot $K\subset S^{3}$ admits an $L$-space surgery for the rational number ${p}/{q}>0$. If $u_{+}(K)=0$, then $K$ is the unknot.
\end{prop}

\begin{proof}
It is known that if $S_{r_{0}}^{3}(J)$ is an $L$-space for the rational number $r_{0}>0$ and the knot $J$, then $S_{r}^{3}(J)$ is an $L$-space as well for any rational number $r\geq r_{0}$ (see \cite[Section~1.3]{MR2249248} for the discussion). Therefore for the knot $K\subset S^{3}$, based on our assumption, we can choose an integer $h>0$ such that $S_{h}^{3}(K)$ is an $L$-space and $-h\in \Omega$.

If $u_{+}(K)=0$, then $u_{-}(\overline{K})=0$ where $\overline{K}$ denotes the mirror image of $K$. Since $-h\in \Omega$, by Theorem~\ref{main} we have $$d(S_{-h}^{3}(\overline{K}), \alpha (s))=d(S_{-h}^{3}(O), s)$$ for some affine isomorphism $\alpha : \mathrm{Spin}^{c}(L(-h, 1))\rightarrow \mathrm{Spin}^{c}(S^{3}_{-h}(\overline{K}))$. This implies that
$$d(S_{h}^{3}({K}), \alpha (s))=d(S_{h}^{3}(O), s)$$ under the natural identification $\mathrm{Spin}^{c}(S^{3}_{-h}(\overline{K}))\cong \mathrm{Spin}^{c}(S^{3}_{h}(K))$. 

Summing the above equation over all $s$ and using Theorem~\ref{lemma}, we conclude that $t_{i}=0$ for all $i\in \mathbb{Z}$. Recalling the definition of $t_{i}$, we see that $a_{i}=t_{i-1}+t_{i+1}-2t_{i}=0$ for any $i\geq 1$. It was proved in \cite{MR2168576} that if a knot $J$ admits a positive integral $L$-space surgery, then the Seifert genus of $J$ coincides with the degree of the normalized Alexander polynomial of $J$. Therefore in our case, the Seifert genus of the knot $K$ is zero, which implies that $K$ is the unknot. 

\end{proof}

This proposition implies that if $K\subset S^{3}$ is a non-trivial knot which admits an $L$-space surgery for a rational number ${p}/{q}> 0$, then $K$ can never be unknotted by performing negative crossing changes.

\begin{rem}
We can modify \cite[Corollary~1.6]{MR2168576} to get the following fact. If $K\subset S^{3}$ is a non-trivial knot which admits a positive integral $L$-space surgery, then Ozsv{\'a}th and Szab{\'o}'s invariant $\tau(K)>0$. This fact also implies that $K$ can never be unknotted by performing negative crossing changes.
\end{rem}

\subsubsection{Examples from cabled knots}
Suppose that $K\subset S^{3}$ is a non-trivial knot for which there is a positive integer $r$ such that $S^{3}_{r}(K)$ is an $L$-space. Let $K_{D}:=D_{+}(K)$ be the positive-clasped untwisted Whitehead double of the knot $K$, and $C_{p,q}(K_{D})$ be the $(p,q)$-cabled knot of the knot $K_{D}$ for positive coprime integers $p$ and $q$. See Figure~\ref{fig:f3} for the illustration of the conventions we are using. In this subsection, we apply Theorem~\ref{main} to study the signed unknotting number of $C_{p,1}(K_{D})$ for any $p\geq 1$, and have the following property.

\begin{prop}
\label{prop2}
The knot $C_{p,1}(K_{D})$ has non-zero positive unknotting number.
\end{prop}

There is a formula that expresses the Alexander polynomial of a satellite knot in terms of those of its companion and its pattern, the statement and the proof of which can be found in \cite[Theorem~6.15]{MR1472978}. By this formula we see that the knot $K_{D}$ and the knot $C_{p,1}(K_{D})$ both have trivial Alexander polynomial, for any $p\geq 1$. The knot $C_{p,1}(K_{D})$ is therefore a topologically slice knot based on Freedman's work \cite[Theorem~11.7B]{MR1201584}. Consequently, we have $u^{T}_{\pm}(C_{p,1}(K_{D}))=0$. In the following paragraphs, we show that $u^{S}_{+}(C_{p,1}(K_{D}))$, and therefore $u_{+}(C_{p,1}(K_{D}))$, are non-zero. We remark that Donaldson's diagonalization theorem, used as a tool to investigate the signed unknotting number by Cochran and Lickorish in \cite{MR849471}, appears to be difficult to apply here to $C_{p,1}(K_{D})$.

\begin{figure}
	\centering
		\includegraphics[width=0.5\textwidth]{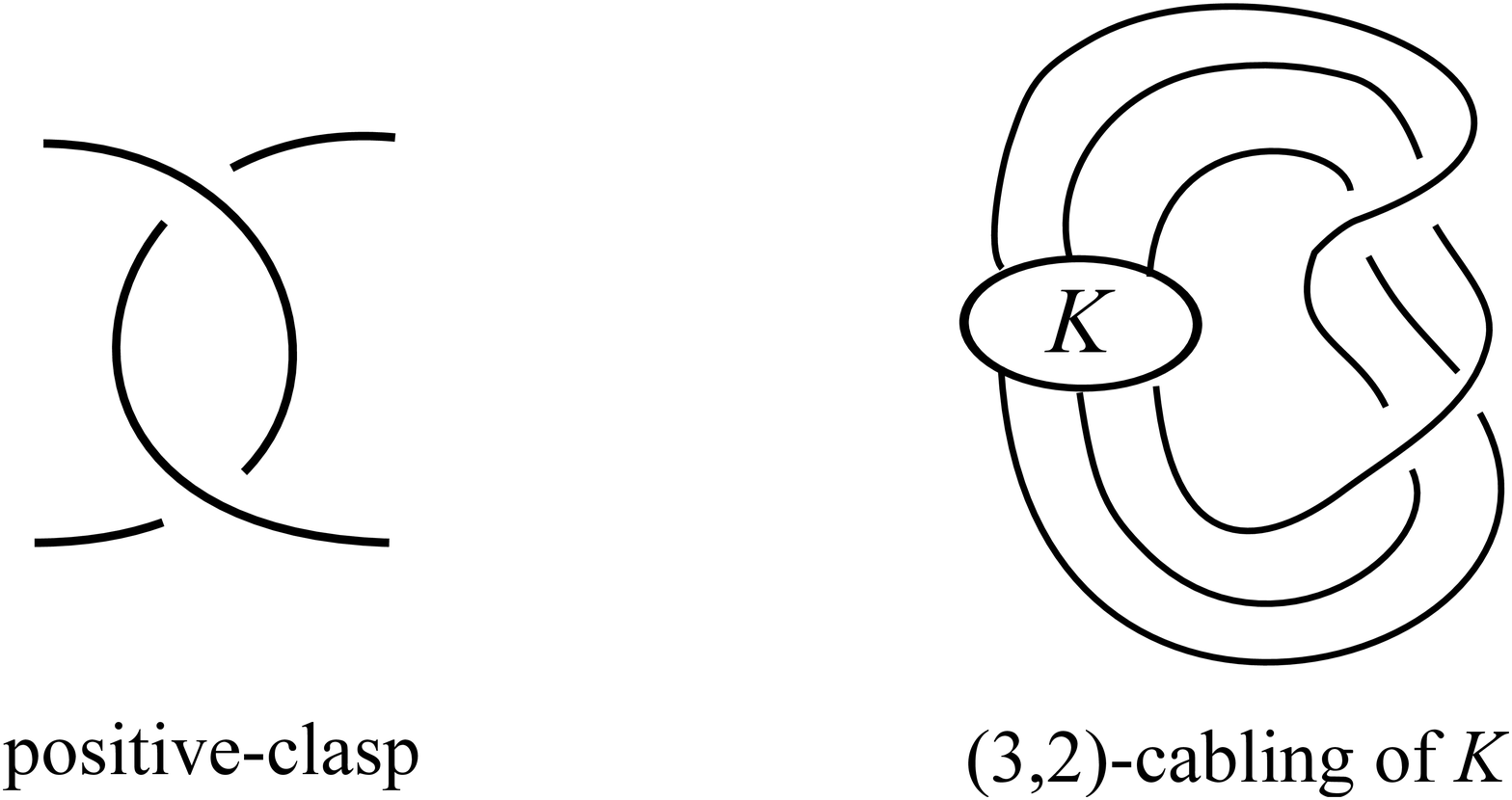}
	\caption{}
	\label{fig:f3}
\end{figure}

\begin{proof}
For any integer $p\geq 1$, it is known from \cite[Corollary~7.3]{MR682725} that 
$$S^{3}_{p}(C_{p,1}(K_{D}))=L(p,1)\sharp S^{3}_{1/p}(K_{D}).$$
Therefore by the additivity of $d$-invariants we see that $$d(S^{3}_{p}(C_{p,1}(K_{D})), s\sharp s_{0})=d(L(p,1), s)+d(S^{3}_{1/p}(K_{D}), s_{0}),$$
 where $s_{0}$ denotes the unique spin$^{c}$-structure over $S^{3}_{1/p}(K_{D})$, and $s$ can be any spin$^c$-structure over $L(p,1)$.
 
 \begin{figure}
	\centering
		\includegraphics[width=0.8\textwidth]{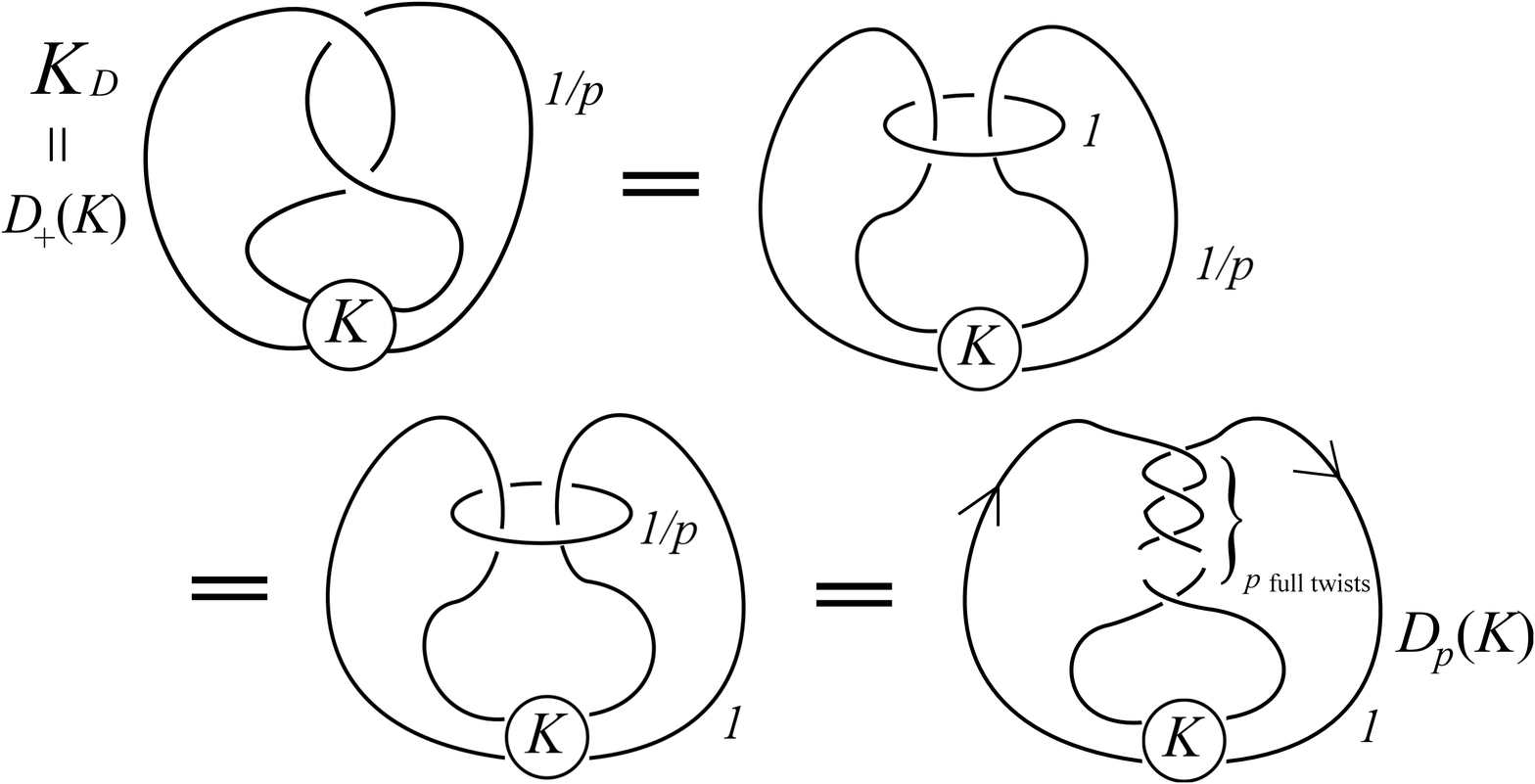}
	\caption{}
	\label{fig:f4}
\end{figure}

Let $D_{p}(K)$ be the knot as shown in Figure~\ref{fig:f4}. As we can see from Figure~\ref{fig:f4}, it holds that $S^{3}_{1/p}(K_{D})\cong S_{1}^{3}(D_{p}(K))$. On the other hand, the knot $D_{p}(K)$ can be transformed into the knot $K_{D}$ by performing only positive crossing changes. By (\ref{3}) we have
$$d(S^{3}_{1/p}(K_{D}), s_{0})=d(S_{1}^{3}(D_{p}(K)), s_{0})\leq d(S^{3}_{1}(K_{D}), s_{0}).$$

%Remember that $K$ admits positive integral $L$-space surgery. %The knot Floer homology of the knot $K$ is described in \cite[Theorem~1.2]{MR2168576}. Then using formulas in \cite{MR2372849}, we can obtain the knot Floer homology for the knot $K_{D}$. The technique introduced in \cite{Peters} tells us how to calculate the $d$-invariant of $S^{3}_{1}(K_{D})$. 
We assign the calculation of $d(S^{3}_{1}(K_{D}), s_{0})$ to the appendix (see Proposition~\ref{c}).
As a conclusion, we have $d(S^{3}_{1}(K_{D}), s_{0})=-2$. See Appendix A for more details about the calculation. Therefore it holds that $$d(S^{3}_{p}(C_{p,1}(K_{D})), s\sharp s_{0})\leq d(L(p,1), s)-2< d(L(p,1), s)$$ for any $s\in \mathrm{Spin}^{c}(L(p,1))$, which implies that $$d(S^{3}_{-p}(\overline{C_{p,1}(K_{D})}), s\sharp s_{0})> d(L(-p,1), s)$$ for any $s\in \mathrm{Spin}^{c}(L(-p,1))$. Since $-p\in \Omega$, then by Theorem~\ref{main} we conclude that $u_{+}({C_{p,1}(K_{D})})=u_{-}(\overline{C_{p,1}(K_{D})})\neq 0$. Namely, $C_{p,1}(K_{D})$ can never be unknotted by only performing negative crossing changes.

\end{proof}
%\begin{rem}
%The signature of the knot $C_{p,1}(K_{m,n})$ for any $p\geq 1$ is zero. But we can see that the Ozsv{\'a}th and Szab{\'o}'s $\tau$ invariant of $C_{p,1}(K_{m,n})$ is $p$, which can be obtained by combining Hedden's formula in \cite{MR2372849} for Whitehead double knots with Hom's formula in \cite{hom} for cable knots. This also implies that $C_{p,1}(K_{m,n})$ can never be unknotted by only performing negative crossing changes. We remark that the signed unknotting number of $C_{p,1}(K_{m,n})$ has yet not been studied before.
% \end{rem}

\bibliographystyle{siam}
\bibliography{5223}

\begin{appendix}
\section{}
The purpose of this appendix is to prove the following property. We follow the method used in \cite{MR2955197}.
\begin{prop}
\label{c}
Let $K\subset S^{3}$ be a non-trivial knot which admits a positive integral $L$-space surgery, and $K_{D}:=D_{+}(K)$ be the positive clasped untwisted  Whitehead double of $K$. Then $d(S^{3}_{1}(K_{D}), s_{0})=-2$.
\end{prop}

\subsection{Preliminaries}
For an oriented rational homology 3-sphere $Y$ and a spin$^{c}$-structure $s\in \mathrm{Spin}^{c}(Y)$, there are several versions of Heegaard Floer homology. The most general Heegaard Floer complex is denoted by $CF^{\infty}(Y, s)$, which is generated by $[x, i]\in I_{s}\times \mathbb{Z}$ for a finite set $I_{s}$. By identifying $[x, i]$ with $U^{-i}x$, the chain group can be thought of as the free $\mathbb{F}[U, U^{-1}]$-module generated by elements in $I_{s}$. Here $U$ is a formal polynomial variable. The chain complex $CF^{\infty}(Y, s)$ has a grading called Maslov grading, which works as the homological grading. Multiplying by $U$ (respectively, $U^{-1}$) decreases (respectively, increases) the Maslov grading by two. The differential on $CF^{\infty}(Y, s)$ decreases the Maslov grading by one and respects the ascending filtration induced by the map $A:I_{s}\times \mathbb{Z}\rightarrow  \mathbb{Z}$ sending $[x, i]$ to $i$. 

Denote the homology of $CF^{\infty}(Y, s)$ by $HF^{\infty}(Y,s)$. It was showed in \cite[Theorem 10.1]{MR2113020} that if $Y$ is a rational homology 3-sphere, then $HF^{\infty}(Y,s)\cong \mathbb{F}[U, U^{-1}]$ for any $\mathrm{spin}^{c}$-structure $s\in\mathrm{Spin}^{c}(Y)$.

Ozsv{\'a}th and Szab{\'o} also defined the complexes $CF^{-}(Y,s)$, $CF^{+}(Y,s)$ and $\widehat{CF}(Y,s)$. Here $CF^{-}(Y,s)$ is a subcomplex of $CF^{\infty}(Y,s)$ generated by $[x, i]$ with $i<0$, and $CF^{+}(Y,s)$ is the quotient complex $CF^{\infty}(Y,s)/CF^{-}(Y,s)$, and thus is generated by $[x, i]$ with $i\geq 0$. The third complex, $\widehat{CF}(Y,s)$, is generated by $[x, 0]$, and has the induced differential from $CF^{\infty}(Y, s)$. It was proved in \cite{MR2113019} that the homology groups of these four complexes are topological invariants of the pair $(Y,s)$. Note that all four homology groups are also modules over $\mathbb{F}[U]$, with the module structure naturally induced from that of $CF^{\infty}(Y,s)$.

Now we focus on the case when $Y=S^{3}$ and $s=s_{0}$, the unique spin$^{c}$-structure over $S^{3}$. The generators of $CF^{\infty}(S^{3}, s_{0})$ are of the form $[x,i]\in I\times \mathbb{Z}$. Ozsv{\' a}th and Szab{\'o} \cite{MR2065507}, and independently Rasmussen \cite{MR2704683} found that a knot $L\subset S^{3}$ induces a filtration, called Alexander filtration, to $CF^{\infty}(S^{3}, s_{0})$, and then defined the filtered chain complex $CFK^{\infty}(S^{3}, L)$. 

Given a generator $[x, i]\in CF^{\infty}(S^{3}, s_{0})$, the Alexander filtration of $[x, i]$ is defined as
$$F([x, i]):=\dfrac{1}{2}(\langle c_{1}(\underline{\mathfrak{s}}(x)), [S]\rangle+2i).$$
Here $\underline{\mathfrak{s}}(x)\in \mathrm{Spin}^{c}(S^{3}_{0}(L))$ (a map $\underline{\mathfrak{s}}:I\rightarrow \mathrm{Spin}^{c}(S^{3}_{0}(L))$ was defined in \cite{MR2065507} and \cite{MR2704683}), and $[S]\in H_{2}(S^{3}_{0}(L); \mathbb{Z})$ is the homology class arising from extending a Seifert surface of $L$ by the meridional disk of the surgery torus. 

The complex $CFK^{\infty}(S^{3}, L)$ can be regarded as a module generated by $[x, i, j]\in I\times \mathbb{Z}\times \mathbb{Z}$ with $j=F([x, i])$. It has a natural $\mathbb{Z}\oplus \mathbb{Z}$ filtration induced by the map $A: I\times \mathbb{Z}\times \mathbb{Z} \rightarrow \mathbb{Z}\oplus \mathbb{Z}$ sending a generator $[x,i,j]$ to $(i,j)$. By identifying $[x, i, j]$ with $U^{-i}[x, 0, j-i]$, the chain group $CFK^{\infty}(S^{3}, L)$ can be thought of as the free $\mathbb{F}[U, U^{-1}]$-module generated by elements of the form $[x, 0, F([x, 0])]$. The Maslov grading on $CFK^{\infty}(S^{3}, L)$ is determined by requiring that $\operatorname{gr}([x, i, j])=\operatorname{gr}([x, i])$, where $\operatorname{gr}([x, i])$ is the Maslov grading in $CF^{\infty}(S^{3}, s_{0})$. The differential of 
$CFK^{\infty}(S^{3}, L)$ decreases Maslov grading by one and respects the $\mathbb{Z}\oplus \mathbb{Z}$ filtration induced by the map $A$, i.e., $A(\partial ([x, i, j]))\leq A([x, i, j])$, where ``$\leq$" is the standard partial order on $\mathbb{Z}\oplus \mathbb{Z}$. See Figure~\ref{fig:f9} for the case when $L$ is the right-handed trefoil knot.

The homology of the filtered chain complex $CFK^{\infty}(S^{3}, L)$ is a topological invariant of $L$, which is called the $\infty$-version of the knot Floer homology of $L$. Consider the subquotient complex $C\lbrace i=0\rbrace\subset CFK^{\infty}(S^{3}, L)$ generated by those generators of the form $[x, 0, j]$. Then $C\lbrace i=0\rbrace$ is again a filtered chain complex with the Alexander filtration induced by $[x, 0, j]\rightarrow j$, and the differential on $C\lbrace i=0\rbrace$ respects this filtration. Let $\mathcal{F}(K, j)\subset C\lbrace i=0\rbrace$ be the filtered subcomplex generated by those generators $[x, 0, m]$ with $m\leq j$. The associated graded homology $\widehat{HFK}(S^{3}, L, j):=H(\mathcal{F}(K, j)/\mathcal{F}(K, j-1))$ is called the hat version of the knot Floer homology of $L$. It is often denoted by $\widehat{HFK}_{l}(S^{3}, L, j)$ the subgroup of $\widehat{HFK}(S^{3}, L)$ whose elements have Maslov grading $l$ and Alexander filtration $j$.

 \begin{figure}
	\centering
		\includegraphics[width=0.35\textwidth]{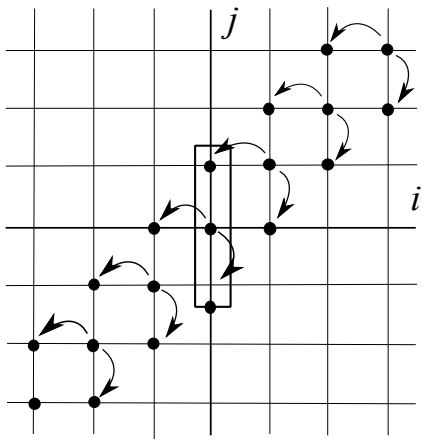}
	\caption{Let $T$ be the right-handed trefoil knot. There is a pointed Heegaard diagram for $(S^{3}, T)$, for which the complex $CFK^{\infty}(S^{3}, T)$ is described as in the figure. The dots represent generators of $CFK^{\infty}(S^{3}, T)$. Generators on the top (respectively middle and bottom) diagonal are of the form $[x, i, i+1]$ (respectively $[y, i, i]$ and $[z, i, i-1]$). Curved arrows indicate the differential of $CFK^{\infty}(S^{3}, T)$. From the figure, we see that $\partial ([y, i, i]=[x, i-1, i]+[z, i, i-1])$. The filtered chain complex $C\lbrace i=0\rbrace$ (encircled by the rectangle) in this case is generated by $[x, 0, 1]$, $[y, 0, 0]$ and $[z, 0, -1]$, and the differential goes from $[y, 0, 0]$ to $[z, 0, -1]$. The associated graded homology $\widehat{HFK}(S^{3}, T)$ thus has rank three.} 
	\label{fig:f9}
\end{figure}

If we forget the Alexander filtration on $CFK^{\infty}(S^{3}, L)$ and $C\lbrace i=0\rbrace$, the homology groups of these two chain complexes are isomorphic to ${HF}^{\infty}(S^{3})\cong \mathbb{F}[U, U^{-1}]$ and $\widehat{HF}(S^{3})\cong \mathbb{F}$. Then Ozsv{\'a}th and Szab{\'o}'s $\tau$-invariant of a knot is defined as follows.
$$\tau (K):=\min \lbrace j \vert\iota_{*}: H(\mathcal{F}(K, j))\rightarrow \widehat{HF}(S^{3}) \text{ is non-trivial }\rbrace.$$ The homomorphism $\iota_{*}$ is induced from the inclusion map. 

Given a knot $L\subset S^{3}$ and a positive integer $r$, we can construct a cobordism from $S^{3}$ to $S^{3}_{r}(K)$ by attaching a 2-handle of framing $r$ along $L$. Let $s_{0}\in \mathrm{Spin}^{c}(S^{3}_{r}(K))$ be the unique $\mathrm{spin}^{c}$-structure that extends to a $\mathrm{spin}^{c}$-structure $t_{0}$ over the cobordism, satisfying $\langle c_{1}(t_{0}), [S_{r}]\rangle +r=0$. Here $[S_{r}]$ is the homology class of a surface in the cobordism obtained by closing off a Seifert surface of $L$ by the core of the 2-handle. 

The following theorem was known.

\begin{theo}[\cite{MR2065507}, Theorem~4.4]
\label{cal}
Let $L\subset S^{3}$ be a knot of Seifert genus $g$. Then there is a chain homotopy equivalence of graded complexes over $\mathbb{F}[U]$,
$$CF^{+}_{l}(S^{3}_{r}(L), s_{0})\cong C_{l+(1-r)/4}\lbrace \max{} (i, j)\geq 0\rbrace,$$ for any integer $r\geq 2g-1$. Here $ C\lbrace \max{} (i, j)\geq 0\rbrace$ is the quotient complex of $CFK^{\infty}(S^{3}, L)$ generated by those generators $[x, i, j]$ with $\max{} (i, j)\geq 0$.
\end{theo}
 
According to the discussion in \cite{Peters} we have $d(S^{3}_{1}(L), s_{0})=d(S^{3}_{r}(L), s_{0})-(r-1)/4$ for any positive integer $r$. In order to prove Proposition~\ref{c}, it is sufficient to calculate $d(S^{3}_{r}(L), s_{0})$ for some sufficiently large $r$. The calculation can be, indeed, fulfilled by applying Theorem~\ref{cal}. 

\subsection{Proof of Proposition~\ref{c}} 

We recall a lemma in \cite{MR2704683}, which is well known in reducing a filtered complex.

\begin{lemma}[Lemma~5.1 in \cite{MR2704683}]
\label{reducing}
Suppose $(C,\partial)$ is a chain complex freely generated by elements $y_{i}$ over a field, and write $d(y_{i}, y_{j})$ for the $y_{j}$ coordinate of $\partial (y_{i})$. Then if $d(y_{k}, y_{l})=1$, we can define a new chain complex $(\bar{C}, \bar{\partial})$ with generators $\lbrace y_{i}\vert i\neq k, l\rbrace$ and differential $$\bar{\partial}(y_{i})=\partial (y_{i})+d(y_{i}, y_{l})\partial(y_{k}).$$ The chain complex $(\bar{C}, \bar{\partial})$ is chain homotopy equivalent to the first one. Moreover, the chain homotopy equivalence $\pi: C\rightarrow \bar{C}$ is the projection, and the equivalence $\iota : \bar{C}\rightarrow C$ sends $y_{i}$ to $y_{i}-d(y_{i}, y_{l})\partial(y_{k})$.
\end{lemma}

For those knots which admit positive integral $L$-space surgeries, the knot Floer homology of them was described by Ozsv{\'a}th and Szab{\'o} as follows. 
\begin{theo}[Ozsv{\'a}th and Szab{\'o} \cite{MR2168576}]
Suppose that $K\subset S^{3}$ is a knot for which there is a positive integer $r$ such that $S^{3}_{r}(K)$ is an $L$-space. Then there is an increasing sequence of integers $n_{-k}<\cdots <n_{k}$ with the property that $n_{i}=-n_{-i}$, having the following significance. For $-k\leq i \leq k$, let 
\begin{equation*}
\delta_{i}=\begin{cases}
0 & \text{if $i=k$,} \\
\delta_{i+1}-2(n_{i+1}-n_{i})+1 & \text{if $k-i$ is odd,} \\
\delta_{i+1}-1 & \text{if $k-i>0$ is even.}
\end{cases}
\end{equation*}
Then 
\begin{equation*}
\widehat{HFK}_{l}(S^{3}, K, j)=\begin{cases}
\mathbb{Z} & \text{if $j=n_{i}$ and $l=\delta_{i}$ for some $i$,} \\
0 & \text{otherwise.}
\end{cases}
\end{equation*}
\end{theo}

It is easy to see that $\delta_{i}\leq 0$ for any $-k\leq i \leq k$. Namely we have $\widehat{HFK}_{l}(S^{3}, K)=0$ for $l>0$. This fact implies the following lemma.
\begin{lemma}
\label{ruo}
 $H_{l}(\mathcal{F}(K, j))=0$ if $l>0$ for any $ j \in \mathbb{Z}$.
\end{lemma}
\begin{proof}
Since we are working in the field $\mathbb{F}:=\mathbb{Z}/2\mathbb{Z}$, by iteratively applying Lemma~\ref{reducing} we see that $\widehat{HFK}(S^{3}, K)$ can be endowed with a differential $\bar{\partial}$ such that $(\widehat{HFK}(S^{3}, K), \bar{\partial})$ is filtered chain homotopy equivalent to $(C\lbrace i=0 \rbrace, \partial )$. Then the lemma follows from the fact that $\widehat{HFK}_{l}(S^{3}, K)=0$ if $l>0$.
\end{proof}

Remember that $K_{D}$ denotes the positive clasped untwisted Whitehead double of $K$. Hedden in \cite[Theorem~1.2]{MR2372849} described the knot Floer homology of $K_{D}$ in terms of that of $K$.

\begin{theo}[Hedden \cite{MR2372849}]
\label{xi}
Let $K\subset S^{3}$ be a non-trivial knot admitting a positive integral $L$-space surgery. Suppose the Seifert genus of $K$ is $g$. Then
\begin{equation*}
\widehat{HFK}_{l}(S^{3}, K_{D}, j)=
\begin{cases}
\mathbb{F}^{-2}_{(1)}\oplus \mathbb{F}^{2g}_{(0)}\bigoplus_{m=-g}^{g} [H_{l-1}(\mathcal{F}(K, m))]^{2} & \text{  $j=1$,}\\
\mathbb{F}^{-4}_{(0)}\oplus \mathbb{F}^{4g-1}_{(-1)}\bigoplus_{m=-g}^{g} [H_{l}(\mathcal{F}(K, m))]^{4} & \text{  $j=0$,}\\
\mathbb{F}^{-2}_{(-1)}\oplus \mathbb{F}^{2g}_{(-2)}\bigoplus_{m=-g}^{g} [H_{l+1}(\mathcal{F}(K, m))]^{2} & \text{  $j=-1$,}\\
0 & otherwise.
\end{cases} 
\end{equation*} 
%\begin{equation*}
%\widehat{HFK}_{l}(S^{3}, K_{D}, j)=
%\begin{cases}
%\mathbb{F}^{2\tau (K)-2g-2}_{(1)}\oplus \mathbb{F}^{2\tau (K)}_{(0)}\bigoplus_{m=-g}^{g} [H_{l-1}(\mathcal{F}(K, m))]^{2} & \text{  j=1}\\
%\mathbb{F}^{4\tau (K)-4g-4}_{(0)}\oplus \mathbb{F}^{4\tau (K)-1}_{(-1)}\bigoplus_{m=-g}^{g} [H_{l}(\mathcal{F}(K, m))]^{4} & \text{  j=0}\\
%\mathbb{F}^{2\tau (K)-2g-2}_{(-1)}\oplus \mathbb{F}^{2\tau (K)}_{(-2)}\bigoplus_{m=-g}^{g} [H_{l+1}(\mathcal{F}(K, m))]^{2} & \text{  j=-1}\\
%0 & otherwise.
%\end{cases} 
%\end{equation*} 
Here the term $\mathbb{F}_{(l)}^{k}$ contributes to $\widehat{HFK}_{l}(S^{3}, K_{D})$, and if $k<0$ we delete it and add $\mathbb{F}_{(l)}^{-k}$ to the left side of the equations. The Ozsv{\'a}th and Szab{\'o}'s $\tau$-invariant $\tau (K_{D})$ is $1$.

\end{theo}

By Lemma~\ref{reducing}, we see that up to $\mathbb{Z}\oplus \mathbb{Z}$-filtered chain homotopy equivalence, a basis for $\widehat{HFK}(S^{3}, K_{D})$ forms a basis for $CFK^{\infty}(S^{3}, K_{D})$ as an $\mathbb{F}[U, U^{-1}]$-module. From now on, we let $CFK^{\infty}(S^{3}, K_{D})$ denote the representative generated by a basis for $\widehat{HFK}(S^{3}, K_{D})$. Under this setting,
Lemma~\ref{ruo} and Theorem~\ref{xi} imply the following lemma.

\begin{lemma}
\label{lan}
The chain complex $CFK^{\infty}(S^{3}, K_{D})$ satisfies the following properties.
\begin{enumerate}
\item Any chain $[x, i, j]\in CFK^{\infty}_{-2}(S^{3}, K_{D})$ satisfies $i+j\geq -1$ or $(i, j)=(-1, -1)$.
\item There exists a cycle $\rho =[x, 0, 1]\in CFK^{\infty}_{0}(S^{3}, K_{D})$ that is homologous to a generator of $HF_{0}^{\infty}(S^{3})\cong \mathbb{F}$. 
\end{enumerate}
\end{lemma}
\begin{proof}
(i) From Lemma~\ref{ruo} and Theorem~\ref{xi}, we see that any element $[x, 0, 1]$ (respectively $[y, 0, 0]$, $[z, 0, -1]$) in $\widehat{HFK}(S^{3}, K_{D})$ has Maslov grading less than or equal to $1$ (respectively $0$, $-1$). Since the variable $U$ carries Maslov grading $-2$ and $\mathbb{Z}\oplus \mathbb{Z}$-filtration $(-1,-1)$, we see that any cycle $[x, i, j]\in CFK^{\infty}_{-2}(S^{3}, K_{D})$ satisfies $i+j\geq -1$ or $(i, j)=(-1,-1)$.

(ii) Recall that $\widehat{HF}(S^{3})\cong H(C\lbrace i=0\rbrace)\cong H(\widehat{HFK}(S^{3}, K_{D}), \bar{\partial})\cong \mathbb{F}$. Since Ozsv{\'a}th and Szab{\'o}'s $\tau$-invariant $\tau (K_{D})=1$, we see that any element $[y, 0, 0]$ and $[z, 0, -1]$ in $\widehat{HFK}(S^{3}, K_{D})$ are zero in $H(\widehat{HFK}(S^{3}, K_{D}), \bar{\partial})$, and that there exists an element $\rho =[x, 0, 1]\in \widehat{HFK}(S^{3}, K_{D})$ generating $\widehat{HF}(S^{3})$ and therefore having Maslov grading zero. 

Now we show that $\rho$, as a chain in $CFK^{\infty}(S^{3}, K_{D})$, is homologous to a generator of $HF_{0}^{\infty}(S^{3})\cong \mathbb{F}$. Since we are working on the field $\mathbb{F}=\mathbb{Z}/2\mathbb{Z}$, it suffices to show that $\rho$ is a non-trivial cycle in $CFK^{\infty}(S^{3}, K_{D})$.

First we show that $\partial \rho =0$ in $CFK^{\infty}(S^{3}, K_{D})$. Otherwise, there exists a chain $[a, -1, 0]$ such that $\partial \rho =[a, -1, 0]$. Since $\partial^{2} \rho =\partial ([a, -1, 0])=0$, we see that $[a, -1, 0]$ is a cycle in $C\lbrace i=-1\rbrace$ with Maslov grading $-1$. Note that there is no chain of Maslov grading $0$ in $C\lbrace i=-1\rbrace$. Therefore $[a, -1, 0]$ is non-trivial in $H(C\lbrace i=-1\rbrace)$. On the other hand, it is known that $H(C\lbrace i=-1\rbrace)\cong \mathbb{F}$ is supported on Maslov grading $-2$. This contradiction implies that $\partial \rho =0$ in $CFK^{\infty}(S^{3}, K_{D})$.

By Lemma~\ref{reducing}, we can endow $\mathbb{F}[U, U^{-1}]$, the $\mathbb{F}[U, U^{-1}]$-module freely generated by $\rho$, a differential $\bar{\partial}$ such that $(CFK^{\infty}(S^{3}, K_{D}), \partial)$ is chain homotopy equivalent to $(\mathbb{F}[U, U^{-1}], \bar{\partial})$. Note that $HFK^{\infty}(S^{3}, K_{D})\cong \mathbb{F}[U, U^{-1}]$. It is easy to see that the differential $\bar{\partial}$ is trivial, which implies that $\rho$ is non-trivial in $(\mathbb{F}[U, U^{-1}], \bar{\partial})$. On the other hand, the homotopy equivalence $\iota : \mathbb{F}[U, U^{-1}]\rightarrow CFK^{\infty}(S^{3}, K_{D})$ sends $\rho$ to itself in $CFK^{\infty}(S^{3}, K_{D})$, since $\partial \rho =0$ in $CFK^{\infty}(S^{3}, K_{D})$. Therefore $\rho$, as a chain in $CFK^{\infty}(S^{3}, K_{D})$, is a non-trivial cycle. This completes the proof.
\end{proof}

\begin{proof}[Proof of Proposition~\ref{c}.]
Consider the generator $\rho ':=U\rho=[x, -1, 0]\in CFK^{\infty}(S^{3}, K_{D})$ where $\rho$ is the cycle in the second statement of Lemma~\ref{lan}. Observe that $\rho'\in C\lbrace \max (i, j)\geq 0\rbrace$. By equivariance $\rho'$ is in the image of $U^{i}$ for all $i\geq 0$. Moreover Theorem~\ref{cal} allows us to view $\rho'$ as a homology class in $HF^{+}(S^{3}_{r}(K_{D}), s_{0})$.

Note that $CFK^{\infty}(S^{3}, K_{D})$ is finitely generated as an $\mathbb{F}[U, U^{-1}]$-module and that multiplying $U^{-1}$ increases the Maslov grading by $2$. Therefore for any fixed Maslov grading $l\gg 0$, it holds that $C_{l}\lbrace \max (i, j)\geq 0\rbrace \cong CFK^{\infty}_{l}(S^{3}, K_{D})$. Hence $[U^{-k}\rho']$ is a non-zero homology class in either group for sufficiently large $k$. It follows that $$\min \lbrace \operatorname{gr}(U^{k}\rho')\vert [U^{k}\rho']\neq 0\in H( C\lbrace \max (i, j)\geq 0\rbrace ) \rbrace$$ is well-defined and equals $d(S^{3}_{r}(K_{D}), s_{0})-(r-1)/4$.

Remember that $\rho'\in C\lbrace \max (i, j)\geq 0\rbrace$, while $U^{k}\rho'\not \in C\lbrace \max (i, j)\geq 0\rbrace$ for any $k\geq 1$. In the following paragraphs, we show that $[\rho']\neq 0\in H(C\lbrace \max (i, j)\geq 0\rbrace )$. If this is true, we see that $$d(S^{3}_{r}(K_{D}), s_{0})=d(S^{3}_{r}(K_{D}), s_{0})-(r-1)/4=\operatorname{gr}(\rho')=-2.$$

We will abuse the notation ``$\partial$" to denote the differential of any chain complex under consideration. Suppose $[\rho']=0\in H(C\lbrace \max (i, j)\geq 0\rbrace)$. Then there exists a chain $\xi \in C\lbrace \max (i, j)\geq 0\rbrace$ such that $\partial \xi=\rho'$. But since $[\rho']\neq 0 $ in $HFK^{\infty}(S^{3}, K_{D})$, there exists a chain $\eta \in C\lbrace \max (i, j)< 0\rbrace$ such that $\partial \xi=\rho'+\eta$ in $CFK^{\infty}(S^{3}, K_{D})$. Since $\operatorname{gr}(\rho')=-2$, so $\operatorname{gr}(\eta)=-2$ as well. By the first statement of Lemma~\ref{lan}, the chain $\eta$ must have the form $\eta=[y, -1,-1]$. 

On the other hand $\partial^{2}(\xi)=\partial (\rho'+\eta)=0$ in $CFK^{\infty}(S^{3}, K_{D})$. Since $\partial \rho'=0$ (remember that $\rho'$ is a cycle in $CFK^{\infty}(S^{3}, K_{D})$), we see that $\partial \eta =0$ as well in $CFK^{\infty}(S^{3}, K_{D})$.
Since $C\lbrace i=-1\rbrace \cong UC\lbrace i=0\rbrace$ and $U^{-1}\eta=[y, 0, 0]$ is trivial in $H(C\lbrace i=0\rbrace)$, therefore $[\eta]=0\in H(C\lbrace i=-1\rbrace)$. Thus there exists a chain $[\zeta, -1, 0] \in C\lbrace i=-1\rbrace$ such that $\partial ( [\zeta, -1, 0]) =\eta$ in $C\lbrace i=-1\rbrace$, and the Maslov grading of $[\zeta, -1, 0] $ is $-1$.

Note that the differential decreases the Maslov grading by $1$ and preserves the $\mathbb{Z}\oplus\mathbb{Z}$-filtration in $CFK^{\infty}(S^{3}, K_{D})$. We see that $\partial \zeta =\eta$ as well in $CFK^{\infty}(S^{3}, K_{D})$. Therefore $\partial (\xi-\zeta)=\rho'$ in $CFK^{\infty}(S^{3}, K_{D})$, which means $[\rho']=0\in HFK^{\infty}(S^{3}, K_{D})$. This contradicts our choice of $\rho'$. Therefore our initial assumption that $[\rho']=0\in H(C\lbrace \max (i, j)\geq 0\rbrace)$ does not hold. This completes the proof of Proposition~\ref{c}.
\end{proof}

\end{appendix}

\end{document}